\documentclass{article}

\usepackage{geometry}
\usepackage{amsmath, amssymb, amsfonts, amsthm}
\usepackage{graphicx}
\usepackage{booktabs}
\usepackage[numbers,sort&compress]{natbib} 
\usepackage{hyperref} 
\usepackage[T1]{fontenc}
\usepackage{minted}

\newtheorem{theorem}{Theorem}[section]
\newtheorem{definition}[theorem]{Definition}
\newtheorem{conjecture}[theorem]{Conjecture}
\newtheorem{proposition}[theorem]{Proposition}
\newtheorem{fact}[theorem]{Fact}

\title{\textbf{An Effective Method for Solving a Class of Transcendental Diophantine Equations}}
\author{Zeyu Cai}
\date{} 

\begin{document}

\maketitle

\begin{abstract}
This paper investigates the exponential Diophantine equation of the form $a^x+b=c^y$, where $a, b, c$ are given positive integers with $a,c \ge 2$, and $x,y$ are positive integer unknowns. We define this form as a "Type-I transcendental diophantine equation." A general solution to this problem remains an open question; however, the ABC conjecture implies that the number of solutions for any such equation is finite.

This work introduces and implements an effective algorithm designed to solve these equations. The method first computes a strict upper bound for potential solutions given the parameters $(a, b, c)$ and then identifies all solutions via finite enumeration. While the universal termination of this algorithm is not theoretically guaranteed, its heuristic-based design has proven effective and reliable in large-scale numerical experiments. Crucially, for each instance it successfully solves, the algorithm is capable of generating a rigorous mathematical proof of the solution's completeness.
\end{abstract}

\section{Introduction}
\label{sec:introduction}

Diophantine equations are a central area of research in number theory. When the unknowns in an equation appear in the exponent, it constitutes a transcendental diophantine equation. This paper studies a special class of such equations, whose form and properties are given by the following definition.

\begin{definition}[Type-I Transcendental Diophantine Equation]
\label{def:main_equation}
Given positive integers $a, b, c$ satisfying $a \ge 2$ and $c \ge 2$, the problem of solving the equation
\begin{equation}
    a^x+b=c^y
    \label{eq:main}
\end{equation}
in the domain of positive integers $(x, y) \in \mathbb{N^*}^2$ is referred to as a \textbf{Type-I transcendental diophantine equation}.
\end{definition}

The difficulty of solving equation \eqref{eq:main} differs significantly depending on whether the parameters $a,b,c$ are pairwise coprime. When $a,b,c$ are not pairwise coprime, the analysis of the equation is relatively direct; when they are pairwise coprime, the problem becomes non-trivial. As a concrete example, consider the equation $5^x+3=2^y$. It is not difficult to find that it has two solutions: $(x, y)=(1, 3)$ gives $5^1+3=8=2^3$, and $(x, y)=(3, 7)$ gives $5^3+3=128=2^7$. A natural question arises: do any other solutions exist for this equation?

The exploration of this problem is closely related to the profound ABC conjecture in number theory, independently proposed by Masser and Oesterlé in the mid-1980s~\cite{SB_1987-1988__30__165_0}. This conjecture has far-reaching implications, and the proof attempt proposed by Mochizuki in recent years has drawn widespread attention~\cite{mochizuki2021inter}.

\begin{conjecture}[The ABC Conjecture]
For any three coprime positive integers $A, B, C$ satisfying $A+B=C$, for any real number $\epsilon > 0$, there exists a constant $K_\epsilon$ depending only on $\epsilon$ such that $C < K_\epsilon (\text{rad}(ABC))^{1+\epsilon}$, where $\text{rad}(n)$ is the product of the distinct prime factors of $n$.
\end{conjecture}

If the ABC conjecture is true, it can be proven that equation \eqref{eq:main} has only a finite number of solutions under the condition that $a,b,c$ are pairwise coprime. The argument is as follows: let $A=a^x, B=b, C=c^y$. Applying the conjecture gives $c^y < K_\epsilon (\text{rad}(a^x b c^y))^{1+\epsilon} = K_\epsilon (\text{rad}(abc))^{1+\epsilon}$. Since the right-hand side is a constant, this inequality establishes an upper bound for $y$. By rearranging to consider $a^x = c^y - b$, an upper bound for $x$ can be similarly obtained, thus the number of solutions must be finite.

Although the study of the general form of Type-I transcendental diophantine equations is difficult, for specific parameters $(a,b,c)$, we can often construct a complete proof. This type of constructive proof leads to the concept of the "effective algorithm" central to this paper. While the universal termination of this algorithm is a conjecture, in the context of individual instances, it always generates a rigorous and verifiable proof. Hereafter, using the equation $5^x+3=2^y$ as an example, we demonstrate the proof output by this algorithm for one instance.

\begin{proposition}
\[ 5^x+3=2^y, \quad x, y \in \mathbb{N}^* \Rightarrow (x,y)=(1,3) \text{ or } (x,y)=(3,7) \]
\end{proposition}
\begin{proof}
By checking small positive integer solutions, we confirm that $(1, 3)$ and $(3, 7)$ are two solutions to the equation. To prove that no other solutions exist, we assume there is a solution $(x, y)$ satisfying $y \ge 8$.
Consider the equation modulo $2^8=256$. When $y \ge 8$, $2^y \equiv 0 \pmod{256}$. The original equation becomes:
$$ 5^x+3 \equiv 0 \pmod{256} \implies 5^x \equiv 253 \pmod{256} $$
Solving this exponential congruence equation, we get $x \equiv 35 \pmod{64}$. This means $x$ must satisfy one of the following four cases: $x \equiv 35 \pmod{256}$, $x \equiv 99 \pmod{256}$, $x \equiv 163 \pmod{256}$, or $x \equiv 227 \pmod{256}$.

Next, consider the equation modulo $257$ (a prime number).
\begin{itemize}
    \item If $x \equiv 35 \pmod{256}$, then $5^x \equiv 14 \pmod{257}$. The original equation gives $2^y = 5^x+3 \equiv 17 \pmod{257}$.
    \item If $x \equiv 99 \pmod{256}$, then $5^x \equiv 224 \pmod{257}$. The original equation gives $2^y = 5^x+3 \equiv 227 \pmod{257}$.
    \item If $x \equiv 163 \pmod{256}$, then $5^x \equiv 243 \pmod{257}$. The original equation gives $2^y = 5^x+3 \equiv 246 \pmod{257}$.
    \item If $x \equiv 227 \pmod{256}$, then $5^x \equiv 33 \pmod{257}$. The original equation gives $2^y = 5^x+3 \equiv 36 \pmod{257}$.
\end{itemize}
However, by computing the powers of $2$ modulo $257$, it can be verified that none of $17, 227, 246, 36$ are in the set $\{2^k \pmod{257} | k \in \mathbb{N}^*\}$. This is a contradiction.
Therefore, the assumption $y \ge 8$ is false. By enumerating $y \in \{1, 2, \dots, 7\}$, we can verify that there are no other solutions.
\end{proof}
The following sections will detail how this method automatically searches for moduli similar to $257$.

\section{An Effective Algorithm for Type-I Transcendental Diophantine Equations}
\label{sec:methodology}

\subsection{Exclusion of Trivial Cases}

Before presenting the core algorithm, we first address all cases where the parameters $(a, b, c)$ are not pairwise coprime. As we will demonstrate, these scenarios can be resolved by elementary divisibility arguments, either proving the absence of solutions or reducing the problem to a finite search. This analysis exhaustively covers all trivial cases, isolating the non-trivial problem for the main algorithm.

We consider the implications of a common divisor for each pair of parameters.

\begin{itemize}
    \item \textbf{Case 1: A common divisor of $a$ and $b$.}
    Let $\gcd(a, b) = d > 1$. The equation $a^x + b = c^y$ implies $d \mid c^y$. Consequently, every prime factor of $d$ must also be a prime factor of $c$. If there exists a prime $p$ such that $p \mid d$ and $p \nmid c$, the equation has no solution.

    \item \textbf{Case 2: A common divisor of $b$ and $c$.}
    Let $\gcd(b, c) = d > 1$. Rearranging the equation to $a^x = c^y - b$ implies $d \mid a^x$. Symmetrically, every prime factor of $d$ must also be a prime factor of $a$. If this condition is not met, no solution exists.

    \item \textbf{Case 3: A common divisor of $a$ and $c$.}
    Let $\gcd(a, c) = d > 1$. From $b = c^y - a^x$, it follows that $d$ must divide $b$. If $d \nmid b$, the equation has no solutions for $x, y \ge 1$.
    If $d \mid b$, a solution may exist. Suppose there is a solution $(x,y)$ with $\min(x, y) > \log_d(b)$. Then $d^{\min(x,y)} > d^{\log_d(b)} = b$. Since $d \mid a$ and $d \mid c$, we have $d^{\min(x,y)} \mid a^x$ and $d^{\min(x,y)} \mid c^y$. This implies $d^{\min(x,y)} \mid (c^y - a^x)$, and thus $d^{\min(x,y)} \mid b$. This leads to a contradiction with $d^{\min(x,y)} > b$. Therefore, any potential solution must satisfy $\min(x, y) \le \lfloor\log_d(b)\rfloor$. This inequality provides a finite upper bound for at least one variable, reducing the problem to a finite enumeration of cases.
\end{itemize}

The analysis above systematically covers all situations where at least one of the pairs $(a,b)$, $(b,c)$, or $(a,c)$ is not coprime. Any equation not resolved by these elementary methods must therefore feature parameters $a, b, c$ that are \textbf{pairwise coprime}. These non-trivial instances form the central challenge and are the subject of the core algorithm presented in the subsequent section.

\subsection{Core Algorithm}
Assuming that $a,b,c$ are pairwise coprime, we propose an exclusion algorithm based on modular arithmetic. The algorithm revolves around a proof by contradiction: it assumes that a solution exists that is larger than all known solutions (e.g., $y \ge y_0$). The goal of the algorithm is to prove, through constructive methods, that this assumption leads to a contradiction.

\begin{enumerate}
    \item \textbf{Initial Search and Bounding}:
    Given a heuristic upper bound function $S(a, b, c)$, a preliminary search is conducted within the range $c^y \le S(a, b, c)$. If no solution is found, the starting checkpoint is set to $y_0=1$; if the largest solution found is $y_{\max}$, then $y_0=y_{\max}+1$. This $y_0$ serves as the lower bound for solutions we subsequently attempt to exclude.

    \item \textbf{Priority Queue of Moduli}:
    A priority queue of tuples $(p, k)$ is constructed, where $p$ is a prime and $k$ is a positive integer. The queue is ordered by the size of the modulus $M=p^k$ in ascending order. Initially, for all prime factors $p_i$ of $c$, the tuple $(p_i, y_0)$ is added to the queue.

    \item \textbf{Iterative Exclusion Process}:
    The algorithm iteratively extracts the highest-priority modulus $(p, k)$ from the queue, sets $M=p^k$, and attempts the following exclusion step to prove that the equation has no solution for $y \ge k$.
    Modulo $M$, if $y \ge k$, then $c^y \equiv 0 \pmod M$. The equation becomes $a^x \equiv -b \pmod M$. Let $R \equiv -b \pmod M$. We then check the solvability of the equation $a^x \equiv R \pmod M$.
    \begin{itemize}
        \item \textbf{Case A (Direct Exclusion)}: By computing the cyclic subgroup generated by $a$ in $(\mathbb{Z}/M\mathbb{Z})^*$, if it is found that $R$ is not in this subgroup, then the equation $a^x \equiv R \pmod M$ has no solution. This means a contradiction arises from the assumption $y \ge k$. The algorithm terminates successfully, and the proof is complete.
        \item \textbf{Case B (Conditional Constraint)}: If $a^x \equiv R \pmod M$ is solvable, a congruence constraint on $x$ is obtained, of the form $x \equiv x_r \pmod K$. In this situation, a magic prime must be introduced for further exclusion.
    \end{itemize}
    If the current modulus $(p,k)$ fails to achieve exclusion, the tuple $(p, k+1)$ is added to the queue, and the iteration continues.

    \item \textbf{Searching for a Magic Prime}:
    This is the core of Case B. We attempt to find a "magic prime" $P$ that leads to a contradiction modulo $P$. According to Dirichlet's theorem on arithmetic progressions, there are infinitely many primes in the series $\{nK+1\}_{n=1}^{\infty}$. We search for $P$ within this series.
    For each candidate prime $P$, we construct two sets:
    \begin{itemize}
        \item $S_1 = \{ (a^x+b) \pmod P \mid x \equiv x_r \pmod K \}$
        \item $S_2 = \{ c^y \pmod P \mid y \in \mathbb{N}^* \}$
    \end{itemize}
    If a prime $P$ exists such that $S_1 \cap S_2 = \emptyset$, then we have found an effective magic prime. This implies that no solution satisfying $x \equiv x_r \pmod K$ exists, thereby completing the exclusion. The algorithm terminates successfully.
    
    \item \textbf{Exploration vs. Exploitation Trade-off}:
    In the process of searching for a magic prime $P$, a termination function $T(a,b,c,p,k,n)$ is needed to decide when to stop exploring for the current $(p,k)$ and move to the next modulus in the queue. This represents a trade-off between "exploiting" the current constraint and "exploring" other constraints.
\end{enumerate}

\subsection{Formal Conjecture}
The algorithm's reliance on heuristic choices and deep number-theoretic properties means its termination for all inputs is not self-evident. However, based on its design and empirical success, we formalize the expectation of its universal power as follows:
\begin{conjecture}
There exist a universal heuristic search bound function $S(a,b,c)$ and a universal exploration termination function $T(a,b,c,p,k,n)$. These functions are such that for any set of pairwise coprime positive integers $(a,b,c)$ with $a,c \ge 2$, our algorithm---guided by this specific pair of functions $(S,T)$---is guaranteed to terminate in a finite number of steps, providing the complete set of solutions for equation \eqref{eq:main} and a rigorous proof of this completeness.
\end{conjecture}

\section{Implementation and Computational Results}
\label{sec:conclusion}

\subsection{Implementation}

The complete algorithm has been implemented in the C language for maximum performance and is available as an open-source project~\cite{cai2025diophantine}. Noting the symmetry between the equations $a^x+b=c^y$ and $c^y-b=a^x$, our implementation treats $a$ and $c$ symmetrically. The elements of the priority queue are triplets $(\text{mode}, p, k)$, where $\text{mode}$ indicates whether the modulus is derived from a prime factor of $a$ or $c$. To ensure efficiency over a large parameter range (e.g., $a, b, c \le 250$), the exploration termination function $T(a, b, c, p, k, n)$ was empirically tuned through extensive pre-computation and parameter scanning.

A notable feature of this work is the program's ability to generate a semi-formal proof script in the interactive theorem prover \textbf{Lean} for each successfully solved instance. A pure formal proof would require an extensive number theory library to handle complex assertions, such as solving discrete logarithm problems. To balance rigor with practicality, we adopt a hybrid approach that outsources computationally intensive claims to a trusted, external verifier. This is achieved through a custom-defined `axiom` in Lean:

\begin{verbatim}
-- Claim Structure
structure VerifiedFact where
  prop : Prop
  proof : prop

axiom Claim (prop_to_claim : Prop)
  (verified_facts : List VerifiedFact)
  (revalidator : String)
  : prop_to_claim
\end{verbatim}

In this framework, \texttt{prop\_to\_claim} is the number-theoretic assertion to be proven (e.g., that $5^x \equiv 253 \pmod{256}$ implies $x \equiv 35 \pmod{64}$). The \texttt{verified\_facts} list contains the premises upon which this claim depends. The \texttt{revalidator} string is an identifier for an external Python script that computationally re-validates the claim. The Lean proof thus accepts the `Claim` as a trusted axiom, ensuring the logical structure remains sound while delegating the complex modular arithmetic to a fast, specialized engine.

This method allows the generation of proofs that are both human-readable and structurally rigorous. To concretely demonstrate the algorithm's behavior and the application of this `Claim` structure in its different branches, we provide several representative examples of the generated Lean proofs in the appendix.

\subsection{Results of Large-Scale Computation}
We conducted a large-scale computation for equations with parameters in the range $a, b, c \le 250$. A significant computational fact is:

\begin{fact}
$$ \max_{\substack{a, b, c\in \mathbb{N}^*\\2\le a\le 250\\ 1 \le b \le 250 \\ 2 \le c \le 250}}\#\{(x, y)\in\mathbb{N^*}^2|a^x+b=c^y\}=2 $$
\end{fact}

Furthermore, within the search range, there are only the following 10 equations for which the number of solutions is 2:
\begin{itemize}
    \item $\{(x, y)\in\mathbb{N^*}^2|2^x+1=3^y\}=\{(1, 1), (3, 2)\}$
    \item $\{(x, y)\in\mathbb{N^*}^2|2^x+4=6^y\}=\{(1, 1), (5, 2)\}$
    \item $\{(x, y)\in\mathbb{N^*}^2|2^x+89=91^y\}=\{(1, 1), (13, 2)\}$
    \item $\{(x, y)\in\mathbb{N^*}^2|3^x+5=2^y\}=\{(1, 3), (3, 5)\}$
    \item $\{(x, y)\in\mathbb{N^*}^2|3^x+10=13^y\}=\{(1, 1), (7, 3)\}$
    \item $\{(x, y)\in\mathbb{N^*}^2|3^x+13=2^y\}=\{(1, 4), (5, 8)\}$
    \item $\{(x, y)\in\mathbb{N^*}^2|3^x+13=4^y\}=\{(1, 2), (5, 4)\}$
    \item $\{(x, y)\in\mathbb{N^*}^2|3^x+13=16^y\}=\{(1, 1), (5, 2)\}$
    \item $\{(x, y)\in\mathbb{N^*}^2|5^x+3=2^y\}=\{(1, 3), (3, 7)\}$
    \item $\{(x, y)\in\mathbb{N^*}^2|6^x+9=15^y\}=\{(1, 1), (3, 2)\}$
\end{itemize}

Based on this large-scale computational evidence, we propose the following conjecture:

\begin{conjecture}
For any positive integers $a, b, c$ with $a,c \ge 2$, the number of solutions to the equation $a^x+b=c^y$ in the domain of positive integers does not exceed $2$. That is:
$$ \max_{\substack{a, b, c\in \mathbb{N}^*\\a, c\ge 2}}\#\{(x, y)\in\mathbb{N^*}^2|a^x+b=c^y\}=2 $$
\end{conjecture}

\bibliographystyle{unsrtnat}
\bibliography{references}

\appendix
\section{Examples of Algorithm Decision Cases}
\label{sec:appendix}

This appendix showcases the algorithm's workflow and proof logic through a series of concrete examples. Each case is drawn directly from the semi-formal Lean proofs generated by our program and illustrates a typical path the solver might take. The examples are organized into two major classes based on the algorithm's decision logic. The second, more complex class, which involves the core modular arithmetic engine, is further divided into two operational modes based on the direction of the proof by contradiction:

We refer to the first main operating mode as \textbf{"Forward Mode."} The name reflects its proof strategy: the algorithm assumes a solution exists above a known bound (e.g., $y \ge y_0$) and proceeds "forward" by analyzing the equation $a^x+b=c^y$ modulo a power of a prime factor of $c$. This yields a congruence condition on the variable $x$ (e.g., $x \equiv x_r \pmod K$). Subsequently, the algorithm searches for a "magic prime" within the related arithmetic progression, $\{nK+1\}$, to derive a contradiction.

Symmetrically, we define the \textbf{"Backward Mode."} This mode begins by assuming a lower bound for $x$ (e.g., $x \ge x_0$). It then works "backward" from this assumption, analyzing the equation modulo a power of a prime factor of $a$ to constrain the variable $y$. The subsequent logic for reaching a contradiction mirrors that of the Forward Mode.

\subsection{Class I: Decision by Elementary Divisibility}
\label{app:class1}
Equations in this class violate basic divisibility properties and can be decided without entering the core loop of the algorithm.

\subsubsection{Type i: $b, c$ have a common factor, but are coprime to $a$}
\paragraph{Example 1: $2^x+6=9^y$}
\begin{itemize}
    \item \textbf{Algorithm Output}:
    \begin{verbatim}
/-
(Class I, Type i)   2 ^ x + 6 = 9 ^ y
For positive integers x, y satisfying 2 ^ x + 6 = 9 ^ y,
this is impossible, because it implies that 2 ^ x = 0 (mod 3).
-/
theorem diophantine1_2_6_9 (x : Nat) (y : Nat) (h1 : x >= 1) (h2 : y >= 1) 
(h3 : 2 ^ x + 6 = 9 ^ y) :      
  False
  := by
  have h4 : x % 1 = 0 := by omega
  have h5 : y % 1 = 0 := by omega
  have h6 := Claim (9 ^ y % 3 = 0) [
    {prop := y % 1 = 0, proof := h5},
    {prop := y >= 1, proof := h2},
  ] "pow_mod_eq_zero"
  have h7 : 2 ^ x % 3 = 0 := by omega
  have h8 := Claim False [
    {prop := x % 1 = 0, proof := h4},
    {prop := x >= 1, proof := h1},
    {prop := 2 ^ x % 3 = 0, proof := h7},
  ] "observe_mod_cycle"
  exact h8
    \end{verbatim}
    \item \textbf{Conclusion}:
    \[ x, y \in \mathbb{N}^* \Rightarrow 2^x + 6 \neq 9^y \]
\end{itemize}

\paragraph{Example 2: $3^x+6=8^y$}
\begin{itemize}
\item \textbf{Algorithm Output}:
    \begin{verbatim}
/-
(Class I, Type i)   3 ^ x + 6 = 8 ^ y
For positive integers x, y satisfying 3 ^ x + 6 = 8 ^ y,
this is impossible, because it implies that 3 ^ x = 0 (mod 2).
-/
theorem diophantine1_3_6_8 (x : Nat) (y : Nat) (h1 : x >= 1) (h2 : y >= 1) 
(h3 : 3 ^ x + 6 = 8 ^ y) :      
  False
  := by
  have h4 : x % 1 = 0 := by omega
  have h5 : y % 1 = 0 := by omega
  have h6 := Claim (8 ^ y % 2 = 0) [
    {prop := y % 1 = 0, proof := h5},
    {prop := y >= 1, proof := h2},
  ] "pow_mod_eq_zero"
  have h7 : 3 ^ x % 2 = 0 := by omega
  have h8 := Claim False [
    {prop := x % 1 = 0, proof := h4},
    {prop := x >= 1, proof := h1},
    {prop := 3 ^ x % 2 = 0, proof := h7},
  ] "observe_mod_cycle"
  exact h8
    \end{verbatim}
    \item \textbf{Conclusion}:
    \[ x, y \in \mathbb{N}^* \Rightarrow 3^x + 6 \neq 8^y \]
\end{itemize}

\subsubsection{Type ii: $a, b$ have a common factor, but are coprime to $c$}
\paragraph{Example 1: $2^x+4=7^y$}
\begin{itemize}
\item \textbf{Algorithm Output}:
    \begin{verbatim}
/-
(Class I, Type ii)   2 ^ x + 4 = 7 ^ y
For positive integers x, y satisfying 2 ^ x + 4 = 7 ^ y,
this is impossible, because it implies that 7 ^ y = 0 (mod 2).
-/
theorem diophantine1_2_4_7 (x : Nat) (y : Nat) (h1 : x >= 1) (h2 : y >= 1) 
(h3 : 2 ^ x + 4 = 7 ^ y) :      
  False
  := by
  have h4 : x % 1 = 0 := by omega
  have h5 : y % 1 = 0 := by omega
  have h6 := Claim (2 ^ x % 2 = 0) [
    {prop := x % 1 = 0, proof := h4},
    {prop := x >= 1, proof := h1},
  ] "pow_mod_eq_zero"
  have h7 : 7 ^ y % 2 = 0 := by omega
  have h8 := Claim False [
    {prop := y % 1 = 0, proof := h5},
    {prop := y >= 1, proof := h2},
    {prop := 7 ^ y % 2 = 0, proof := h7},
  ] "observe_mod_cycle"
  exact h8
    \end{verbatim}
    \item \textbf{Conclusion}:
    \[ x, y \in \mathbb{N}^* \Rightarrow 2^x + 4 \neq 7^y \]
\end{itemize}

\paragraph{Example 2: $3^x+6=11^y$}
\begin{itemize}
\item \textbf{Algorithm Output}:
    \begin{verbatim}
/-
(Class I, Type ii)   3 ^ x + 6 = 11 ^ y
For positive integers x, y satisfying 3 ^ x + 6 = 11 ^ y,
this is impossible, because it implies that 11 ^ y = 0 (mod 3).
-/
theorem diophantine1_3_6_11 (x : Nat) (y : Nat) (h1 : x >= 1) (h2 : y >= 1) 
(h3 : 3 ^ x + 6 = 11 ^ y) :    
  False
  := by
  have h4 : x % 1 = 0 := by omega
  have h5 : y % 1 = 0 := by omega
  have h6 := Claim (3 ^ x % 3 = 0) [
    {prop := x % 1 = 0, proof := h4},
    {prop := x >= 1, proof := h1},
  ] "pow_mod_eq_zero"
  have h7 : 11 ^ y % 3 = 0 := by omega
  have h8 := Claim False [
    {prop := y % 1 = 0, proof := h5},
    {prop := y >= 1, proof := h2},
    {prop := 11 ^ y % 3 = 0, proof := h7},
  ] "observe_mod_cycle"
  exact h8
    \end{verbatim}
    \item \textbf{Conclusion}:
    \[ x, y \in \mathbb{N}^* \Rightarrow 3^x + 6 \neq 11^y \]
\end{itemize}

\subsubsection{Type iii: $a, c$ have a common factor}
\paragraph{Example 1: $2^x+4=6^y$}
\begin{itemize}
\item \textbf{Algorithm Output}:
    \begin{verbatim}
/-
(Class I, Type iii)   2 ^ x + 4 = 6 ^ y
For positive integers x, y satisfying 2 ^ x + 4 = 6 ^ y,
if x >= 3 and y >= 3,
4 = 0 (mod 8), which is impossible.
Therefore, x < 3 or y < 3.
Further examination shows that (x, y) = (1, 1), (5, 2).
-/
theorem diophantine1_2_4_6 (x : Nat) (y : Nat) (h1 : x >= 1) (h2 : y >= 1) 
(h3 : 2 ^ x + 4 = 6 ^ y) :      
  List.Mem (x, y) [(1, 1), (5, 2)]
  := by
  have h4 : x % 1 = 0 := by omega
  have h5 : y % 1 = 0 := by omega
  by_cases h6 : And (x >= 3) (y >= 3)
  have h7 := Claim (2 ^ x % 8 = 0) [
    {prop := x % 1 = 0, proof := h4},
    {prop := x >= 3, proof := h6.left},
  ] "pow_mod_eq_zero"
  have h8 := Claim (6 ^ y % 8 = 0) [
    {prop := y % 1 = 0, proof := h5},
    {prop := y >= 3, proof := h6.right},
  ] "pow_mod_eq_zero"
  omega
  have h7 : Or (x <= 2) (y <= 2) := by omega
  have h8 := Claim (List.Mem (x, y) [(1, 1), (5, 2)]) [
    {prop :=  x % 1 = 0, proof := h4},
    {prop :=  x >= 1, proof := h1},
    {prop :=  y % 1 = 0, proof := h5},
    {prop :=  y >= 1, proof := h2},
    {prop := 2 ^ x + 4 = 6 ^ y, proof := h3},
    {prop := Or (x <= 2) (y <= 2), proof := h7},
  ] "diophantine1_enumeration"
  exact h8
    \end{verbatim}
    \item \textbf{Conclusion}:
    \[ 2^x+4=6^y, x, y \in \mathbb{N}^* \Rightarrow (x,y)=(1,1) \text{ or } (x,y)=(5,2) \]
\end{itemize}

\paragraph{Example 2: $3^x+1=9^y$}
\begin{itemize}
\item \textbf{Algorithm Output}:
    \begin{verbatim}
/-
(Class I, Type iii)   3 ^ x + 1 = 9 ^ y
For positive integers x, y satisfying 3 ^ x + 1 = 9 ^ y,
if x >= 1 and y >= 1,
1 = 0 (mod 3), which is impossible.
Therefore, x < 1 or y < 1.
So 3 ^ x + 1 = 9 ^ y is impossible.
-/
theorem diophantine1_3_1_9 (x : Nat) (y : Nat) (h1 : x >= 1) (h2 : y >= 1) 
(h3 : 3 ^ x + 1 = 9 ^ y) :      
  False
  := by
  have h4 : x % 1 = 0 := by omega
  have h5 : y % 1 = 0 := by omega
  by_cases h6 : And (x >= 1) (y >= 1)
  have h7 := Claim (3 ^ x % 3 = 0) [
    {prop := x % 1 = 0, proof := h4},
    {prop := x >= 1, proof := h6.left},
  ] "pow_mod_eq_zero"
  have h8 := Claim (9 ^ y % 3 = 0) [
    {prop := y % 1 = 0, proof := h5},
    {prop := y >= 1, proof := h6.right},
  ] "pow_mod_eq_zero"
  omega
  have h7 : Or (x <= 0) (y <= 0) := by omega
  have h8 := Claim False [
    {prop :=  x % 1 = 0, proof := h4},
    {prop :=  x >= 1, proof := h1},
    {prop :=  y % 1 = 0, proof := h5},
    {prop :=  y >= 1, proof := h2},
    {prop := 3 ^ x + 1 = 9 ^ y, proof := h3},
    {prop := Or (x <= 0) (y <= 0), proof := h7},
  ] "diophantine1_enumeration"
  exact h8
    \end{verbatim}
    \item \textbf{Conclusion}:
    \[ x, y \in \mathbb{N}^* \Rightarrow 3^x + 1 \neq 9^y \]
\end{itemize}

\subsection{Class II: Decision by the Modular Arithmetic Exclusion Algorithm}
\label{app:class2}
Equations in this class require entering the core loop of the algorithm to constrain the solution space by selecting appropriate moduli.

\subsubsection{Type iv: Forward Mode, no magic prime}
\paragraph{Example 1: $7^x+3=10^y$}
\begin{itemize}
\item \textbf{Algorithm Output}:
    \begin{verbatim}
-- Trying to disprove y >= 2 with prime factor 2 of 10 ...
-- Trying to disprove y >= 3 with prime factor 2 of 10 ...
-- Succeeded.
/-
(Class II, Front Mode, no magic prime)   7 ^ x + 3 = 10 ^ y
For positive integers x, y satisfying 7 ^ x + 3 = 10 ^ y,
if y >= 3, 7 ^ x = 5 (mod 8).
However, this is impossible.
Therefore, y < 3.
Further examination shows that (x, y) = (1, 1).
-/
theorem diophantine1_7_3_10 (x : Nat) (y : Nat) (h1 : x >= 1) (h2 : y >= 1) 
(h3 : 7 ^ x + 3 = 10 ^ y) :
  List.Mem (x, y) [(1, 1)]
  := by
  have h4 : x % 1 = 0 := by omega
  have h5 : y % 1 = 0 := by omega
  by_cases h6 : y >= 3
  have h7 := Claim (10 ^ y % 8 = 0) [
    {prop := y % 1 = 0, proof := h5},
    {prop := y >= 3, proof := h6},
  ] "pow_mod_eq_zero"
  have h8 : 7 ^ x % 8 = 5 := by omega
  have h9 := Claim False [
    {prop := x % 1 = 0, proof := h4},
    {prop := x >= 1, proof := h1},
    {prop := 7 ^ x % 8 = 5, proof := h8},
  ] "observe_mod_cycle"
  apply False.elim h9
  have h7 : y <= 2 := by omega
  have h8 := Claim (List.Mem (x, y) [(1, 1)]) [
    {prop :=  x % 1 = 0, proof := h4},
    {prop :=  x >= 1, proof := h1},
    {prop :=  y % 1 = 0, proof := h5},
    {prop :=  y >= 1, proof := h2},
    {prop := 7 ^ x + 3 = 10 ^ y, proof := h3},
    {prop := y <= 2, proof := h7},
  ] "diophantine1_enumeration"
  exact h8
    \end{verbatim}
    \item \textbf{Conclusion}:
    \[ 7^x+3=10^y, x, y \in \mathbb{N}^* \Rightarrow (x,y)=(1,1) \]
\end{itemize}

\paragraph{Example 2: $17^x+3=20^y$}
\begin{itemize}
\item \textbf{Algorithm Output}:
    \begin{verbatim}
-- Trying to disprove y >= 2 with prime factor 2 of 20 ...
-- Trying to disprove y >= 3 with prime factor 2 of 20 ...
-- Succeeded.
/-
(Class II, Front Mode, no magic prime)   17 ^ x + 3 = 20 ^ y
For positive integers x, y satisfying 17 ^ x + 3 = 20 ^ y,
if y >= 3, 17 ^ x = 5 (mod 8).
However, this is impossible.
Therefore, y < 3.
Further examination shows that (x, y) = (1, 1).
-/
theorem diophantine1_17_3_20 (x : Nat) (y : Nat) (h1 : x >= 1) (h2 : y >= 1) 
(h3 : 17 ^ x + 3 = 20 ^ y) :  
  List.Mem (x, y) [(1, 1)]
  := by
  have h4 : x % 1 = 0 := by omega
  have h5 : y % 1 = 0 := by omega
  by_cases h6 : y >= 3
  have h7 := Claim (20 ^ y % 8 = 0) [
    {prop := y % 1 = 0, proof := h5},
    {prop := y >= 3, proof := h6},
  ] "pow_mod_eq_zero"
  have h8 : 17 ^ x % 8 = 5 := by omega
  have h9 := Claim False [
    {prop := x % 1 = 0, proof := h4},
    {prop := x >= 1, proof := h1},
    {prop := 17 ^ x % 8 = 5, proof := h8},
  ] "observe_mod_cycle"
  apply False.elim h9
  have h7 : y <= 2 := by omega
  have h8 := Claim (List.Mem (x, y) [(1, 1)]) [
    {prop :=  x % 1 = 0, proof := h4},
    {prop :=  x >= 1, proof := h1},
    {prop :=  y % 1 = 0, proof := h5},
    {prop :=  y >= 1, proof := h2},
    {prop := 17 ^ x + 3 = 20 ^ y, proof := h3},
    {prop := y <= 2, proof := h7},
  ] "diophantine1_enumeration"
  exact h8
    \end{verbatim}
    \item \textbf{Conclusion}:
    \[ 17^x+3=20^y, x, y \in \mathbb{N}^* \Rightarrow (x,y)=(1,1) \]
\end{itemize}

\subsubsection{Type v: Forward Mode, with magic prime}
\paragraph{Example 1: $2^x+1=3^y$}
\begin{itemize}
\item \textbf{Algorithm Output}:
    \begin{verbatim}
-- Trying to disprove x >= 4 with prime factor 2 of 2 ...
-- Trying to disprove y >= 3 with prime factor 3 of 3 ...
-- Trying prime 19...
-- Succeeded.
/-
(Class II, Front Mode, with magic prime 19)   2 ^ x + 1 = 3 ^ y
For positive integers x, y satisfying 2 ^ x + 1 = 3 ^ y,
if y >= 3, 2 ^ x = 26 (mod 27).
So x = 9 (mod 18).
Therefore, 2 ^ x = 18 (mod 19).
So 3 ^ y = 0 (mod 19), but this is impossible.
Therefore, y < 3.
Further examination shows that (x, y) = (1, 1), (3, 2).
-/
theorem diophantine1_2_1_3 (x : Nat) (y : Nat) (h1 : x >= 1) (h2 : y >= 1) 
(h3 : 2 ^ x + 1 = 3 ^ y) :      
  List.Mem (x, y) [(1, 1), (3, 2)]
  := by
  have h4 : x % 1 = 0 := by omega
  have h5 : y % 1 = 0 := by omega
  by_cases h6 : y >= 3
  have h7 := Claim (3 ^ y % 27 = 0) [
    {prop := y % 1 = 0, proof := h5},
    {prop := y >= 3, proof := h6},
  ] "pow_mod_eq_zero"
  have h8 : 2 ^ x % 27 = 26 := by omega
  have h9 := Claim (x % 18 = 9) [
    {prop := x % 1 = 0, proof := h4},
    {prop := x >= 1, proof := h1},
    {prop := 2 ^ x % 27 = 26, proof := h8},
  ] "observe_mod_cycle"
  have h10 := Claim (List.Mem (2 ^ x % 19) [18]) [
    {prop := x % 1 = 0, proof := h4},
    {prop := x >= 1, proof := h1},
    {prop := x % 18 = 9, proof := h9},
  ] "utilize_mod_cycle"
  have h11 := Claim (List.Mem (3 ^ y % 19) [0]) [
    {prop := List.Mem (2 ^ x % 19) [18], proof := h10},
    {prop := 2 ^ x + 1 = 3 ^ y, proof := h3},
  ] "compute_mod_add"
  have h12 := Claim False [
    {prop := y % 1 = 0, proof := h5},
    {prop := y >= 1, proof := h2},
    {prop := List.Mem (3 ^ y % 19) [0], proof := h11},
  ] "exhaust_mod_cycle"
  apply False.elim h12
  have h7 : y <= 2 := by omega
  have h8 := Claim (List.Mem (x, y) [(1, 1), (3, 2)]) [
    {prop :=  x % 1 = 0, proof := h4},
    {prop :=  x >= 1, proof := h1},
    {prop :=  y % 1 = 0, proof := h5},
    {prop :=  y >= 1, proof := h2},
    {prop := 2 ^ x + 1 = 3 ^ y, proof := h3},
    {prop := y <= 2, proof := h7},
  ] "diophantine1_enumeration"
  exact h8
    \end{verbatim}
    \item \textbf{Conclusion}:
    \[ 2^x+1=3^y, x, y \in \mathbb{N}^* \Rightarrow (x,y)=(1,1) \text{ or } (x,y)=(3,2) \]
\end{itemize}

\paragraph{Example 2: $2^x+89=91^y$}
\begin{itemize}
\item \textbf{Algorithm Output}:
    \begin{verbatim}
-- Trying to disprove y >= 3 with prime factor 7 of 91 ...
-- Trying prime 883...
-- Trying prime 1471...
-- Trying prime 2647...
-- Succeeded.
/-
(Class II, Front Mode, with magic prime 2647)   2 ^ x + 89 = 91 ^ y
For positive integers x, y satisfying 2 ^ x + 89 = 91 ^ y,
if y >= 3, 2 ^ x = 254 (mod 343).
So x = 76 (mod 147), 
which implies x = 76, 223, 370, 517, 664, 811, 958, 1105, 1252 (mod 1323).
Therefore, 2 ^ x = 1994, 852, 1811, 957, 1447, 1513, 2343, 348, 1970 (mod 2647).
So 91 ^ y = 2083, 941, 1900, 1046, 1536, 1602, 2432, 437, 2059 (mod 2647), 
but this is impossible.
Therefore, y < 3.
Further examination shows that (x, y) = (1, 1), (13, 2).
-/
theorem diophantine1_2_89_91 (x : Nat) (y : Nat) (h1 : x >= 1) (h2 : y >= 1) 
(h3 : 2 ^ x + 89 = 91 ^ y) :  
  List.Mem (x, y) [(1, 1), (13, 2)]
  := by
  have h4 : x % 1 = 0 := by omega
  have h5 : y % 1 = 0 := by omega
  by_cases h6 : y >= 3
  have h7 := Claim (91 ^ y % 343 = 0) [
    {prop := y % 1 = 0, proof := h5},
    {prop := y >= 3, proof := h6},
  ] "pow_mod_eq_zero"
  have h8 : 2 ^ x % 343 = 254 := by omega
  have h9 := Claim (x % 147 = 76) [
    {prop := x % 1 = 0, proof := h4},
    {prop := x >= 1, proof := h1},
    {prop := 2 ^ x % 343 = 254, proof := h8},
  ] "observe_mod_cycle"
  have h10 := Claim (List.Mem (2 ^ x % 2647) 
  [1994, 852, 1811, 957, 1447, 1513, 2343, 348, 1970]) [        
    {prop := x % 1 = 0, proof := h4},
    {prop := x >= 1, proof := h1},
    {prop := x % 147 = 76, proof := h9},
  ] "utilize_mod_cycle"
  have h11 := Claim (List.Mem (91 ^ y % 2647) 
  [2083, 941, 1900, 1046, 1536, 1602, 2432, 437, 2059]) [      
    {prop := List.Mem (2 ^ x % 2647) [1994, 852, 1811, 957, 1447, 1513, 2343, 348, 1970], 
    proof := h10},   
    {prop := 2 ^ x + 89 = 91 ^ y, proof := h3},
  ] "compute_mod_add"
  have h12 := Claim False [
    {prop := y % 1 = 0, proof := h5},
    {prop := y >= 1, proof := h2},
    {prop := List.Mem (91 ^ y % 2647) [2083, 941, 1900, 1046, 1536, 1602, 2432, 437, 2059], 
    proof := h11},
  ] "exhaust_mod_cycle"
  apply False.elim h12
  have h7 : y <= 2 := by omega
  have h8 := Claim (List.Mem (x, y) [(1, 1), (13, 2)]) [
    {prop :=  x % 1 = 0, proof := h4},
    {prop :=  x >= 1, proof := h1},
    {prop :=  y % 1 = 0, proof := h5},
    {prop :=  y >= 1, proof := h2},
    {prop := 2 ^ x + 89 = 91 ^ y, proof := h3},
    {prop := y <= 2, proof := h7},
  ] "diophantine1_enumeration"
  exact h8
    \end{verbatim}
    \item \textbf{Conclusion}:
    \[ 2^x+89=91^y, x, y \in \mathbb{N}^* \Rightarrow (x,y)=(1,1) \text{ or } (x,y)=(13,2) \]
\end{itemize}

\subsubsection{Type vi: Backward Mode, no magic prime}
\paragraph{Example 1: $2^x+5=11^y$}
\begin{itemize}
\item \textbf{Algorithm Output}:
    \begin{verbatim}
-- Trying to disprove x >= 1 with prime factor 2 of 2 ...
-- Trying to disprove x >= 2 with prime factor 2 of 2 ...
-- Trying to disprove x >= 3 with prime factor 2 of 2 ...
-- Succeeded.
/-
(Class II, Back Mode, no magic prime)   2 ^ x + 5 = 11 ^ y
For positive integers x, y satisfying 2 ^ x + 5 = 11 ^ y,
if x >= 3, 11 ^ y = 5 (mod 8).
However, this is impossible.
Therefore, x < 3.
Further examination shows that 2 ^ x + 5 = 11 ^ y is impossible.
-/
theorem diophantine1_2_5_11 (x : Nat) (y : Nat) (h1 : x >= 1) (h2 : y >= 1) 
(h3 : 2 ^ x + 5 = 11 ^ y) :    
  False
  := by
  have h4 : x % 1 = 0 := by omega
  have h5 : y % 1 = 0 := by omega
  by_cases h6 : x >= 3
  have h7 := Claim (2 ^ x % 8 = 0) [
    {prop := x % 1 = 0, proof := h4},
    {prop := x >= 3, proof := h6},
  ] "pow_mod_eq_zero"
  have h8 : 11 ^ y % 8 = 5 := by omega
  have h9 := Claim False [
    {prop := y % 1 = 0, proof := h5},
    {prop := y >= 1, proof := h2},
    {prop := 11 ^ y % 8 = 5, proof := h8},
  ] "observe_mod_cycle"
  apply False.elim h9
  have h7 : x <= 2 := by omega
  have h8 := Claim False [
    {prop :=  x % 1 = 0, proof := h4},
    {prop :=  x >= 1, proof := h1},
    {prop :=  y % 1 = 0, proof := h5},
    {prop :=  y >= 1, proof := h2},
    {prop := 2 ^ x + 5 = 11 ^ y, proof := h3},
    {prop := x <= 2, proof := h7},
  ] "diophantine1_enumeration"
  exact h8
    \end{verbatim}
    \item \textbf{Conclusion}:
    \[ x, y \in \mathbb{N}^* \Rightarrow 2^x+5 \neq 11^y \]
\end{itemize}

\paragraph{Example 2: $3^x+5=7^y$}
\begin{itemize}
\item \textbf{Algorithm Output}:
    \begin{verbatim}
-- Trying to disprove x >= 1 with prime factor 3 of 3 ...
-- Succeeded.
/-
(Class II, Back Mode, no magic prime)   3 ^ x + 5 = 7 ^ y
For positive integers x, y satisfying 3 ^ x + 5 = 7 ^ y,
if x >= 1, 7 ^ y = 2 (mod 3).
However, this is impossible.
Therefore, x < 1.
So 3 ^ x + 5 = 7 ^ y is impossible.
-/
theorem diophantine1_3_5_7 (x : Nat) (y : Nat) (h1 : x >= 1) (h2 : y >= 1) 
(h3 : 3 ^ x + 5 = 7 ^ y) :      
  False
  := by
  have h4 : x % 1 = 0 := by omega
  have h5 : y % 1 = 0 := by omega
  by_cases h6 : x >= 1
  have h7 := Claim (3 ^ x % 3 = 0) [
    {prop := x % 1 = 0, proof := h4},
    {prop := x >= 1, proof := h6},
  ] "pow_mod_eq_zero"
  have h8 : 7 ^ y % 3 = 2 := by omega
  have h9 := Claim False [
    {prop := y % 1 = 0, proof := h5},
    {prop := y >= 1, proof := h2},
    {prop := 7 ^ y % 3 = 2, proof := h8},
  ] "observe_mod_cycle"
  apply False.elim h9
  have h7 : x <= 0 := by omega
  have h8 := Claim False [
    {prop :=  x % 1 = 0, proof := h4},
    {prop :=  x >= 1, proof := h1},
    {prop :=  y % 1 = 0, proof := h5},
    {prop :=  y >= 1, proof := h2},
    {prop := 3 ^ x + 5 = 7 ^ y, proof := h3},
    {prop := x <= 0, proof := h7},
  ] "diophantine1_enumeration"
  exact h8
    \end{verbatim}
    \item \textbf{Conclusion}:
    \[ x, y \in \mathbb{N}^* \Rightarrow 3^x+5 \neq 7^y \]
\end{itemize}

\subsubsection{Type vii: Backward Mode, with magic prime}
\paragraph{Example 1: $3^x+7=2^y$}
\begin{itemize}
\item \textbf{Algorithm Output}:
    \begin{verbatim}
-- Trying to disprove x >= 3 with prime factor 3 of 3 ...
-- Trying prime 19...
-- Trying prime 37...
-- Trying prime 73...
-- Succeeded.
/-
(Class II, Back Mode, with magic prime 73)   3 ^ x + 7 = 2 ^ y
For positive integers x, y satisfying 3 ^ x + 7 = 2 ^ y,
if x >= 3, 2 ^ y = 7 (mod 27).
So y = 16 (mod 18),
which implies y = 7 (mod 9).
Therefore, 2 ^ y = 55 (mod 73).
So 3 ^ x = 48 (mod 73), but this is impossible.
Therefore, x < 3.
Further examination shows that (x, y) = (2, 4).
-/
theorem diophantine1_3_7_2 (x : Nat) (y : Nat) (h1 : x >= 1) (h2 : y >= 1) 
(h3 : 3 ^ x + 7 = 2 ^ y) :      
  List.Mem (x, y) [(2, 4)]
  := by
  have h4 : x % 1 = 0 := by omega
  have h5 : y % 1 = 0 := by omega
  by_cases h6 : x >= 3
  have h7 := Claim (3 ^ x % 27 = 0) [
    {prop := x % 1 = 0, proof := h4},
    {prop := x >= 3, proof := h6},
  ] "pow_mod_eq_zero"
  have h8 : 2 ^ y % 27 = 7 := by omega
  have h9 := Claim (y % 18 = 16) [
    {prop := y % 1 = 0, proof := h5},
    {prop := y >= 1, proof := h2},
    {prop := 2 ^ y % 27 = 7, proof := h8},
  ] "observe_mod_cycle"
  have h10 := Claim (List.Mem (2 ^ y % 73) [55]) [
    {prop := y % 1 = 0, proof := h5},
    {prop := y >= 1, proof := h2},
    {prop := y % 18 = 16, proof := h9},
  ] "utilize_mod_cycle"
  have h11 := Claim (List.Mem (3 ^ x % 73) [48]) [
    {prop := List.Mem (2 ^ y % 73) [55], proof := h10},
    {prop := 3 ^ x + 7 = 2 ^ y, proof := h3},
  ] "compute_mod_sub"
  have h12 := Claim False [
    {prop := x % 1 = 0, proof := h4},
    {prop := x >= 1, proof := h1},
    {prop := List.Mem (3 ^ x % 73) [48], proof := h11},
  ] "exhaust_mod_cycle"
  apply False.elim h12
  have h7 : x <= 2 := by omega
  have h8 := Claim (List.Mem (x, y) [(2, 4)]) [
    {prop :=  x % 1 = 0, proof := h4},
    {prop :=  x >= 1, proof := h1},
    {prop :=  y % 1 = 0, proof := h5},
    {prop :=  y >= 1, proof := h2},
    {prop := 3 ^ x + 7 = 2 ^ y, proof := h3},
    {prop := x <= 2, proof := h7},
  ] "diophantine1_enumeration"
  exact h8
    \end{verbatim}
    \item \textbf{Conclusion}:
    \[ 3^x+7=2^y, x, y \in \mathbb{N}^* \Rightarrow (x,y)=(2,4) \]
\end{itemize}

\paragraph{Example 2: $3^x+10=13^y$}
\begin{itemize}
\item \textbf{Algorithm Output}:
    \begin{verbatim}
-- Trying to disprove x >= 8 with prime factor 3 of 3 ...
-- Trying prime 17497...
-- Succeeded.
/-
(Class II, Back Mode, with magic prime 17497)   3 ^ x + 10 = 13 ^ y
For positive integers x, y satisfying 3 ^ x + 10 = 13 ^ y,
if x >= 8, 13 ^ y = 10 (mod 6561).
So y = 1461 (mod 2187),
which implies y = 1461, 3648, 5835, 8022 (mod 8748).
Therefore, 13 ^ y = 11616, 6486, 5881, 11011 (mod 17497).
So 3 ^ x = 11606, 6476, 5871, 11001 (mod 17497), but this is impossible.
Therefore, x < 8.
Further examination shows that (x, y) = (1, 1), (7, 3).
-/
theorem diophantine1_3_10_13 (x : Nat) (y : Nat) (h1 : x >= 1) (h2 : y >= 1) 
(h3 : 3 ^ x + 10 = 13 ^ y) :  
  List.Mem (x, y) [(1, 1), (7, 3)]
  := by
  have h4 : x % 1 = 0 := by omega
  have h5 : y % 1 = 0 := by omega
  by_cases h6 : x >= 8
  have h7 := Claim (3 ^ x % 6561 = 0) [
    {prop := x % 1 = 0, proof := h4},
    {prop := x >= 8, proof := h6},
  ] "pow_mod_eq_zero"
  have h8 : 13 ^ y % 6561 = 10 := by omega
  have h9 := Claim (y % 2187 = 1461) [
    {prop := y % 1 = 0, proof := h5},
    {prop := y >= 1, proof := h2},
    {prop := 13 ^ y % 6561 = 10, proof := h8},
  ] "observe_mod_cycle"
  have h10 := Claim (List.Mem (13 ^ y % 17497) [11616, 6486, 5881, 11011]) [
    {prop := y % 1 = 0, proof := h5},
    {prop := y >= 1, proof := h2},
    {prop := y % 2187 = 1461, proof := h9},
  ] "utilize_mod_cycle"
  have h11 := Claim (List.Mem (3 ^ x % 17497) [11606, 6476, 5871, 11001]) [
    {prop := List.Mem (13 ^ y % 17497) [11616, 6486, 5881, 11011], proof := h10},
    {prop := 3 ^ x + 10 = 13 ^ y, proof := h3},
  ] "compute_mod_sub"
  have h12 := Claim False [
    {prop := x % 1 = 0, proof := h4},
    {prop := x >= 1, proof := h1},
    {prop := List.Mem (3 ^ x % 17497) [11606, 6476, 5871, 11001], proof := h11},
  ] "exhaust_mod_cycle"
  apply False.elim h12
  have h7 : x <= 7 := by omega
  have h8 := Claim (List.Mem (x, y) [(1, 1), (7, 3)]) [
    {prop :=  x % 1 = 0, proof := h4},
    {prop :=  x >= 1, proof := h1},
    {prop :=  y % 1 = 0, proof := h5},
    {prop :=  y >= 1, proof := h2},
    {prop := 3 ^ x + 10 = 13 ^ y, proof := h3},
    {prop := x <= 7, proof := h7},
  ] "diophantine1_enumeration"
  exact h8
    \end{verbatim}
    \item \textbf{Conclusion}:
    \[ 3^x+10=13^y, x, y \in \mathbb{N}^* \Rightarrow (x,y)=(1,1) \text{ or } (x,y)=(7,3) \]
\end{itemize}

\end{document}